\documentclass{article}

\usepackage{amsmath}
\usepackage{amsfonts}
\usepackage{amssymb}
\usepackage{amsthm}
\usepackage{enumerate}
\usepackage{bbm}
\usepackage{hyperref}

\newtheorem{theorem}{Theorem}[section]
\newtheorem{lemma}[theorem]{Lemma}
\newtheorem{claim}[theorem]{Claim}

\newtheorem{corollary}[theorem]{Corollary}

\newtheorem{observation}[theorem]{Observation}

\newcommand{\suchthat}{\;\ifnum\currentgrouptype=16 \middle\fi|\;}

\newcommand{\summ}{\displaystyle\sum}

\newcommand{\F}{\mathbb{F}}

\newcommand*{\defeq}{\stackrel{\text{def}}{=}}

\begin{document}
\title{Cauchy-Davenport Theorem for linear maps: Simplification and Extension}
\author{
John Kim\thanks{Department of Mathematics, Rutgers University.  Research supported in part by NSF Grant Number DGE-1433187. {\tt jonykim@math.rutgers.edu}.} \and
Aditya Potukuchi\thanks{Department of Computer Science, Rutgers University. {\tt aditya.potukuchi@cs.rutgers.edu}.}}
\maketitle

\begin{abstract} 
We give a new proof of the Cauchy-Davenport Theorem for linear maps given by Herdade et al., (2015) in~\cite{HKK}. This theorem gives a lower bound on the size of the image of a linear map on a grid. Our proof is purely combinatorial and offers a partial insight into the range of parameters not handled in~\cite{HKK}.
\end{abstract}

\section{Introduction}

Let $\F_p$ be the field containing $p$ elements, where $p$ is a prime, and let $A,B \subseteq \F_p$. The Cauchy-Davenport Theorem gives a lower bound on the size of the sumset $A + B \defeq \{a + b \suchthat a \in A,b \in B\}$ (for more on sumsets, see, for example,~\cite{TV}). The size of the sumset can be thought of as the size of the image of the linear map $(x,y) \rightarrow x + y$, where $x \in A$, and $y \in B$. Thus the theorem can be restated as follows:

%We refer the interested reader to~\cite{TaoandVu} for more detail on Cauchy-Davenport Theorem and sumsets. \\

\begin{theorem}[Cauchy-Davenport Theorem]
Let $p$ be a prime, and $L:\F_p \times \F_p \rightarrow \F_p$ be a linear map that takes $(a,b)$ to $a+b$. For $A,B \subseteq \F_p$, Let $L(A,B)$ be the image of $L$ on $A \times B$. Then, 

$$
|L(A,B)| \geq \min(|A|+|B|-1,p)
$$
\end{theorem}

In~\cite{HKK}, this notion was extended to study the sizes of images of general linear maps on product sets. A lower bound was proved using the polynomial method (via a nonstandard application of  the Combinatorial Nullstellensatz~\cite{Alon}). In this paper, we give a simpler, and combinatorial proof of the same using just the Cauchy-Davenport Theorem. \\

Notation: For a linear map $L : \F_p^{n} \rightarrow \F_p^m$, and for $S_1, S_2, \ldots S_n \subseteq \F_p$, we use $L(S_1, S_2 \ldots S_n)$ to denote the image of $L$ on $S_1 \times S_2 \times \cdots S_n$. The \emph{support} of a vector is the set of nonzero entries in the vector. A \emph{min-support vector} in a set $V$ of vectors is a nonzero vector of minimum support size in $V$.

\begin{theorem}[Main Theorem]
\label{the:main}
Let $p$ be a prime, and $L:\F_p^{m+1} \rightarrow \F_p^m$ be a linear map of rank $m$. Let $A_1, A_2, \ldots A_{m+1} \subseteq \F_p$ with $|A_i| = k_i$. Further, suppose that $\min_i(k_i) + \max_i(k_i) < p$. Let $S$ be the support of $\ker(L)$, and $S' = [n] \setminus S$. Then

$$
|L(A_1,A_2, \ldots, A_n)| \geq \left( \prod_{j \in S'}k_j \right) \cdot \left( \prod_{i\in S}k_i - \prod_{i \in S}(k_i - 1) \right)
$$
\end{theorem}

As noted in~\cite{HKK}, this bound is tight for every $m$ and $p$. We restrict our theorem to study only maps from $\F_p^{m+1}$ to $\F_p^m$ of rank $m$ for two reasons mainly:(1) It is simpler to state, and contains the tight case and (2) We are unable prove any better bounds if the rank is not $m$. It is not clear to us what the correct bound for the general case is. \\

We also show the following result for the size of the image for certain full rank linear maps from $\F_p^n \rightarrow \F_p^{n-1}$ when the size of the sets it is evaluated on are all large enough.

\begin{theorem}
\label{cover}
Let $L:\mathbb{F}_p^n \rightarrow \mathbb{F}_p^{n-1}$ be a linear map given by $L(x_1,\ldots x_n) = (x_1 + x_n, x_2 + x_n \ldots x_{n-1} + x_n)$. Let $S_1, \ldots S_n \subseteq \F_p$ with $|S_i| = k$ for $i \in [n]$ such that $k > \frac{(n-1)p}{n}$, then $|L(S_1, \ldots S_n)| = p^{n-1}$ (i.e., $L(S_1, \ldots S_n) = \F_p^{n-1}$).
\end{theorem}

The theorems do not, however, give tight bounds for all set sizes, for example if $\min_i |A_i| > p/2$. It would be interesting to obtain a tight bound even for the simple linear map $(x,y,z) \rightarrow (x+z, y+z)$ on the product set $A_1 \times A_2 \times A_3 \subseteq \F_p^3$ which holds for all sizes of the $A_i$'s.

\section{The Theorem}

\subsection{The Main Lemma}

The idea is that since the size of the image is invariant under row operations of $L$, we perform row operations to isolate a `hard' part, which gives the main part of the required lower bound \\

Our proof proceeds by induction on the dimension of the linear map. The base case is given by the Cauchy Davenport Theorem.

\begin{lemma}
\label{lem:main}
Let $L:\mathbb{F}_p^n \rightarrow \mathbb{F}_p^{n-1}$ be a linear map such that $L(x_1,\ldots, x_n) = (x_1 + x_n, x_2 + x_n, \ldots, x_{n-1} + x_n)$. Let $S_1, \ldots S_n \subseteq \F_p$ with $|S_i| = s_i$ such that $\min_i(s_i) + \max_i(s_i) \leq p+1$. Then $|L(S_1, \ldots S_n)| \geq \prod_{i=1}^n s_i-\prod_{i=1}^n (s_i-1)$
\end{lemma}

\begin{proof}
We use the shorthand notation $|L| \defeq |L(S_1,S_2 \ldots S_n)|$. W.L.O.G, let $S_1$ be such that $|S_1| = \min_{i \in [n-1]}(|S_i|)$. \\

A preliminary observation is that $|S_1| + |S_n| \leq p+1$, and therefore, by the Cauchy-Davenport Theorem, 
\begin{equation} \label{CD}
|S_1 + S_n| \geq s_1 + s_n - 1
\end{equation}

The proof proceeds by induction on $n$. If $n = 2$, the result $|L| \geq s_1 \cdot s_2 - (s_1-1) \cdot (s_2-1) = s_1 + s_2 - 1$ is given by the Cauchy-Davenport Theorem. \\

For every $a \in \F_p$, we have $T_a \defeq \{x_n \in S_n \suchthat \exists x_1 \in S_1, x_1 + x_n = a\}$, and $t_a \defeq |T_a|$. We now look at the restricted linear map $L|_{x_1 + x_n = a}$. In this case, the induction is on sets $S_2, \ldots S_{n-1} \times T_a$. This is equivalent to restricting $S_n$ to the set $T_a$, and dropping $S_1$, since for every $x_n \in S_n$, there is a unique $x_1 \in S_1$ such that $x_1 + x_n = a$. \\

We first observe that the conditions are satisfied, i.e., $\min_i(|S_i|) + \max_i(|S_i|) \leq p + 1$, since $t_a \leq \min(|S_1|,|S_n|)$. Also the resulting linear map is of the same form, i.e., $L|_{x_1 + x_n = a}(x_2,\ldots x_n) = (x_2 + x_n \ldots x_{n-1} + x_n)$. (In reality, $L|_{x_1 + x_n = a}$ is a map from $\F_p^n$ to $\F_p^{n-1}$, given by $L|_{x_1 + x_n = a}(x_1,x_2,\ldots x_n) = (a, x_2 + x_n \ldots x_{n-1} + x_n)$ but we drop the first coordinate because it is fixed, i.e., $a$)\\

By induction hypothesis, the number of points in the image of $L_{x_1 + x_n = a}$ is at least:

$$
\left( \prod_{i=2}^{n-1}s_i \right) t_a - \left( \prod_{i=2}^{n-1}(s_i - 1) \right)(t_a - 1)
$$ 

Summing over all $a \in \mathbb{F}_p$, we get a bound on the number of points in the image:

\begin{eqnarray*}
|L| &\geq& \sum_{a \in \F_p, t_a \neq 0} \left( \left( \prod_{i=2}^{n-1}s_i \right) t_a - \left( \prod_{i=2}^{n-1}(s_i - 1) \right)(t_a - 1) \right) \\
&=& \left( \prod_{i=2}^{n-1}s_i \right)\sum_{a \in \F_p}t_a -  \left( \prod_{i=2}^{n-1}(s_i - 1) \right)\sum_{a\in \F_p, t_a \neq 0}(t_a - 1) \\
&\geq& \prod_{i=1}^{n}s_i - \prod_{i=1}^n(s_i - 1)
\end{eqnarray*}

The last inequality comes from observing that $\sum_{a \in \F_p}t_a = s_1s_n$, and an upper bound on $\sum_{a \in \F_p, t_a \neq 0}(t_a - 1)$, by using~\ref{CD}. We have $\sum_{a \in \F_p, t_a \neq 0}(t_a - 1) = \sum_{a \in \mathbb{F}}t_a - \sum_{a \in \F_p}\mathbbm{1}_{t_a \neq 0} = \sum_{a \in \mathbb{F}}t_a + |S_1 + S_n| \leq s_1s_n - (s_1 + s_n - 1)$.

\end{proof}

\subsection{Arriving at the Main Theorem}

The first step in arriving at the main theorem is exactly as in~\cite{HKK}. For completeness, we describe it here. The idea is to transform a general linear map into a specific form, without reducing the size of the image (in fact, here it remains the same). This step is very intuitive, but describing it requires some setup. \\

Let $L:\F_p^{m+1} \rightarrow \F_p^m$ be an $\F_p$-linear map of rank $m$. Let $v$ be a non-zero min-support vector of $\ker(L)$. So, we have $Lv = 0$. The main observation is that under row operations, two quantities remain unchanged: the size of the image of $L$, and the size of the support of the min-support vector in the kernel. \\

Let $r_1, \ldots r_m$ be the rows, and $c_1, c_2, \ldots c_{m+1}$ be the columns of associated to $L$ with respect to the standard basis. We show that one can perform elementary row operations, and some column operations on $L$ while preserving the size of the image.

\begin{lemma}
\label{lem:linop}
The size of the image of $L$ does not change under

\begin{enumerate}
\item Elementary row operations.
\item Scaling any column $c_i$ by some $d \in \F_p \setminus \{0\}$ and scaling every element of $A_i$ by $d$.
\item Swapping any two columns $c_i$ and $c_j$, and swapping sets $A_i$ and $A_j$.
\end{enumerate}
\end{lemma}

\begin{proof}

We prove this by considering each given operation separately.

\begin{enumerate}
\item Suppose $L'$ was obtained from $L$ by elementary row operations. There is an invertible linear map $M$ such that $M \cdot L = L'$. This gives the bijection from every vector $v$ in the image of $L$, to the vector $M \cdot v$ in the image of $L'$.

%\item Suppose $L'$ is obtained from $L$ by switching rows $r_i$ and $r_j$. In this case For every vector $(u_1, \ldots, u_i, \ldots, u_j, \ldots, u_m) \in L(A_1, A_2, \ldots ,A_{m+1})$, we have that $(u_1, \ldots, u_j, \ldots, u_i, \ldots, u_m) \in L'(A_1, A_2, \ldots, A_{m+1})$. This map is invertible.

%\item Suppose $L'$ is obtained from $L$ by scaling row $i$ by $d \in \F_p \setminus \{0\}$. In this case For every vector $(u_1, \ldots, u_i, \ldots, u_m) \in L(A_1, A_2, \ldots, A_{m+1})$, we have that $(u_1, \ldots, du_i, \ldots, u_m) \in L'(A_1, A_2, \ldots, A_{m+1})$. This map is invertible

%\item Suppose $L'$ was the linear map obtained from $L$ by doing the row operation: $r_i \leftarrow r_i + r_j$. For every vector $(u_1, \ldots, u_i, \ldots, u_j, \ldots, u_m) \in L(A_1, A_2, \ldots, A_{m+1})$, we have the vector $(u_1, \ldots, u_i + u_j, \ldots, u_j, \ldots, u_m) \in$ $L'(A_1, A_2, \ldots, A_{m+1})$. This map is invertible.

\item Suppose $L'$ was obtained from $L$ by scaling column $c_i$ by $d \in \F_p \setminus \{0\}$, and scaling the set $A_i$ by $d^{-1}$. We map every vector $(u_1, \ldots, u_m) \in$ $L(A_1,\ldots ,A_i, \ldots, A_{m+1})$, to the vector $(u_1, \ldots, u_m) \in L'(A_1,\ldots , d^{-1} \cdot A_i, \ldots, A_{m+1})$. Here $d^{-1} \cdot A_i \defeq \{d^{-1}a_i \suchthat a_i \in A_i\}$ This map is invertible.

\item Suppose $L'$ was obtained from $L$ by switching columns $c_i$ and $c_j$, and swapping the sets $A_i$ and $A_j$. We map every vector $(u_1,\ldots, u_m) \in L(A_1,\ldots ,A_i, \ldots, A_j, \ldots, A_{m+1})$ to the identical vector $(u_1, \ldots, u_m) \in$ 

$ L'(A_1,\ldots ,A_j, \ldots, A_i, \ldots, A_{m+1})$. This map is invertible.
\end{enumerate}

For every given operation, we have a bijection between the images of $L$ before and after the operation.

\end{proof}

\begin{observation}
After the operations stated in Lemma~\ref{lem:linop}, the size of the support of the min-support vector in $\ker(L)$ does not change. 
\end{observation}
To see this, we first observe that the kernel has rank $1$, and is orthogonal to the row span of $L$. Therefore, all nonzero vectors in $\ker(L)$ have the same support. Since, row operations do not change the row span of $L$, the resulting kernel spans the same subspace of $\F^{m+1}$, and therefore, the size of the support of the vectors in $\ker(L)$ does not change. \\

Next, we do the following operations, each of which preserves the size of the image.

\begin{enumerate}
\item Perform row operations so that the last $m$ columns form an identity matrix.

\item Scale the rows so that the first column of every row is $1$.

\item Scale the last $m$ columns so that every nonzero entry in $L$ is $1$.
\end{enumerate}

After we perform these operations, we have a linear map where the first column consists of $1$'s and $0$'s and the remaining $m$ columns form an identity matrix. Let the $S'$ be the set of indices of rows containing $1$'s in the first column. Consider the vector $v = -e_1 + \sum_{i \in S'}e_{i+1}$. This vector has support $|S'| + 1$, and lies in the kernel of $L$. Therefore, $|S| = |S'| + 1$.

\begin{proof}[Proof of Theorem~\ref{the:main}]
Apply the transformation from Lemma~\ref{lem:linop} to $L$ to reduce it to the simple form. Let $S'$ be the set of rows where the first column is nonzero. Consider the restriction of $L$ on the the coordinates given by $S$. By Lemma~\ref{lem:main}, the size of this image is at least $\left( \prod_{i\in S}k_i - \prod_{i \in S}(k_i - 1) \right)$. 

The linear map restricted to the coordinates $[m] \setminus S$ is nothing but the identity map, so the size of the image is $\prod_{i \not \in S} |A_i|$, and is independent of the linear map restricted to $S$. Putting them together, we have the desired result.
\end{proof}

\section{The case when $2k > p+1$}

The proof of Lemma~\ref{lem:main} breaks down when $s_1+s_n > p + 1$ and, unfortunately, we do not know how to fix this issue. Consider, for example, the simplest nontrivial case where $m = 2$, i.e., $L(x,y,z) = (x+z, y+z)$, and we are interested in the size of the image of $L$ on $X \times Y \times Z$, further suppose, for simplicity, that $|X|=|Y|=|Z| = k$. If $k < \frac{p + 1}{2}$, then the above bound holds, and is tight. If $k>\frac{2p}{3}$, then $L$ covers $\F_p^2$, i.e., $|L(A,B,C)| = p^2$. This makes the case in between the interesting one. We conjecture that the correct lower bound is the size of the image of $L$ when $X = Y = Z = \{1,2,\ldots k\}$. Towards this, we are able to prove a partial result (Lemma~\ref{2d}) using the above method.

We will need the following Lemma:

\begin{lemma}
\label{boundtx}
Let $X,Y\subseteq \F_p$ and $t_a = |\{(x,y)\in X\times Y: x+y = a\}|$. Then for every $a \in \F_p$:
$$|X|+|Y|-p\leq t_a \leq \min(|X|,|Y|).$$
\end{lemma}

\begin{proof}
The bounds follow from the fact that $t_a$ can be written as the size of the 
intersection of two sets of sizes $|X|$ and $|Y|$: 
$$t_a = |X \cap (a-Y)|.$$
\end{proof}

Now we state the partial result:

\begin{theorem}
\label{2d}
Let $L:\F_p^3\rightarrow \F_p^2$ be the linear map defined by $L(x,y,z) = (x+z,y+z)$.  
Let $X,Y,Z\subset F_p$ be sets of size $k$, where $k \geq \frac{p+1}{2}$.  
Then we have the following lower bound:
$$|L(X,Y,Z)| \geq \min(p^2 + 3k^2 - (2p+1)k,p^2).$$
\end{theorem}

\begin{proof}
Let $T_a \defeq \{z \in Z \suchthat \exists x \in X, x+z = a\}$, with $t_a\defeq |T_a|$.
Looking at this restriction, $L|_{x+z = a}$, by Cauchy-Davenport Theorem, there are at least $\min(t_a + k-1,p)$ points of $L(X,Y,Z)$ on $L_{x+z=a}(Y,T_a)$. By summing over all $a\in \F_p$, 
we get a lower bound on the size of $L(X,Y,Z)$:

\begin{eqnarray*}
|L(X,Y,Z)| &\geq& \summ_{a\in\F_p}{\min(t_a + k-1,p)} \\
&=& \summ_{a\in\F_p}{\min(t_a,p-k+1)} + p(k-1) \\
&=& \summ_{a:t_a \leq p-k+1}{t_a} + \summ_{a:t_a > p-k+1}{(p-k+1)} + p(k-1).
\end{eqnarray*}

We now want to remove the dependence of the lower bound on the $t_a$ by considering 
the worst case scenario, where the $t_a$ take values that minimize the lower bound.  
First, we observe $\sum_{a\in\F_p}{t_a} = k^2$, a fixed quantity. So to minimize the above lower bound for $|L(X,Y,Z)|$, we need $t_a$ to be maximal for as many $a\in\F_p$ as possible.

By Lemma~\ref{boundtx}, we know that $2k-p\leq t_a \leq k$.  We set $t_a = k$ for as many 
$a\in\F_p$ as possible, and the remainder of the $t_a = 2k-p$.  This gives:

\begin{eqnarray*}
|L(X,Y,Z)| &\geq& \summ_{a:t_a \leq p-k+1}{t_a} + \summ_{a:t_a > p-k+1}{(p-k+1)} + p(k-1) \\
&\geq& k(2k-p) + (p-k)(p-k+1) + p(k-1) \\
& = & 3k^2 + p^2 - (2p-1)k.
\end{eqnarray*}

\end{proof}

As a corollary, we get, independent of theorem~\ref{cover}, the following corrolary:

\begin{corollary}
If the linear map $L$, and the sets $A$, $B$, $C$ were as above, with $|A|=|B|=|C|=k$, and $k > \frac{2p}{3}$, then $L(A,B,C) = p^2$.
\end{corollary}

We would like to point out that at the two extremes, i.e., when $k = \frac{p+1}{2}$, and when $k =\lceil\frac{2p}{3} \rceil$, the above bound matches the `correct' lower bound.

\subsection{Proof of Theorem~\ref{cover}}

We prove theorem~\ref{cover} via a slightly stronger claim
\begin{claim}
\label{lem:cover}
Let $L:\mathbb{F}_p^n \rightarrow \mathbb{F}_p^{n-1}$ be a linear map given by $L(x_1,\ldots x_n) = (x_1 + x_n, x_2 + x_n \ldots x_{n-1} + x_n)$. Let $S_1, \ldots S_n \subseteq \F_p$ with $|S_i| = k$ for $i \in [n-1]$, and $|S_n| = k'$. Further, suppose that $(n-1)k + k' \geq (n-1)p + 1$, then $|L(S_1, \ldots S_n)| = p^{n-1}$.
\end{claim}

\begin{proof}
We prove this by induction on $n$, analogous to Lemma~\ref{lem:main}. The case where $n = 2$ is, again, given by the Cauchy-Davenport Theorem. \\

For $a \in \F_p$, $T_a \defeq \{x_n \in S_n \suchthat \exists x_1 \in S_1, x_1 + x_n = a\}$ with $t_a \defeq |T_a|$. Looking at this restriction of $L$ (i.e., $x_1 + x_n = a$), we have a linear map, $L_{x_1 + x_n = a}$ on the sets $S_2 \times S_3 \times \cdots T_a$, given by the $L_{x_1 + x_n = a}(x_2 ,\ldots x_n) = (x_2 + x_n, \ldots x_{n-1} + x_n)$. (similar to Lemma~\ref{lem:main}, we drop the first coordinate). \\

Here, $|S_i| = k$ for $i = 2,\ldots n-1$, and $|T_a| \geq k+k'-p$, by Lemma~\ref{boundtx}. Further, the required condition holds, i.e.,:

$$(n-2)k + t_a \geq (n-2)k + k + k' - p = (n-1)k + k' - p \geq (n-2)p+1.$$ 

Therefore, by induction hypothesis $|L|_{x_1 + x_n = a}(S_2, \ldots S_{n-1}, T_a)| = p^{n-2}$. Since this holds for every $a \in \F_p$, we have $|L(S_1, \ldots S_n)| = p^{n-1}$.
\end{proof}

In particular, Lemma~\ref{lem:cover} tells that for the linear map $L$ given by $L(x_1,\ldots x_n) = (x_1 + x_n, x_2 + x_n, \ldots, x_{n-1} + x_n)$ on $S_1 \times S_2 \times \cdots S_n$, if $|S_i| \geq \frac{(n-1)p}{n}$, then $L(S_1, \ldots, S_n) = \F_p^{n-1}$.

\section{Acknowledgements}

We would like to thank Swastik Kopparty for the discussions and the many helpful ideas.

\bibliographystyle{alpha}

\end{document}